\documentclass[12pt]{amsart}

\usepackage{enumerate}
\usepackage{amsmath}
\allowdisplaybreaks
\usepackage{mathrsfs}
\usepackage{color}

\usepackage[normalem]{ulem}
\usepackage{amssymb}
\usepackage{hyperref}

\newtheorem{theorem}            {Theorem}[section]
\newtheorem{proposition}        [theorem]{Proposition}

\title{Improved Average Bounds for an Inequality of Alladi, Erdös, and Vaaler}
\author{Cooper O'Kuhn}

\begin{document}
\maketitle

\thispagestyle{empty}

\begin{abstract}
    We prove an averaged version of a claim suspected to be true by Alladi, Erdös, and Vaaler.  Qualitatively, the result states that a divisor sum of a multiplicative function, which obeys certain size constraints, derives most of its value from its ``small" divisors.  
\end{abstract}

\section{Introduction}

Alladi, Erdos, and Vaaler [AEV] suspected that, for every squarefree $n$, if $h$ is a multiplicative function with $0 \leq h(p) \leq c \leq \frac{1}{k - 1}$ for some integer $k \geq 2$ and every prime~$p$, then

\begin{align}
    \sum_{d | n} h(d) \leq (4 + o(1)) \sum_{\substack{d | n \\ d \leq n ^ \frac{1}{k}}} h(d) \label{main conj}
\end{align}

where the $o(1)$ tends to zero as the number of distinct prime divisors $\omega(n)$ tends to infinity. One qualitative interpretation of this is that a divisor sum of a sufficiently small multiplicative function derives most of its value from its small divisors. It can also be thought of as an extension of the classical identity 
$$\sum_{d|n} 1 = 2 \sum_{\substack{d|n \\ d \leq \sqrt{n}}} 1$$
which holds when $n$ is not square, though since squarefree numbers are the chief consideration here, we've omitted that case.

It is shown in \cite{ALLADI1989183} that the bound $h(p) \leq \frac{1}{k - 1}$ is sharp, since if $h$ is taken to be bigger, the inequality \eqref{main conj} breaks down for $n$ with certain constraints on its prime divisors. This inequality was proven unconditionally with the constant $4$ replaced with $2k$ in \cite{ALLADI1989183}, which was subsequently improved to $k$ by Soundararajan in \cite{SOUNDARARAJAN1992225}. It has since been unknown whether $k$-dependence could be removed. Letting $\mu(n)$ denote the Möbius function, we prove the following averaged result.

\begin{theorem}
\label{thm:constant on primes case}
Let $h$ be a multiplicative function such that $h(p) = c$ for some $0 \leq c < \frac{1}{k - 1}$ and for all prime $p$. For $k \geq 3$, we have
\begin{align}
    \sum_{n \leq x} \mu^2(n) \sum_{d|n} h(d) &\leq (2 + o(1))\sum_{n \leq x} \mu^2(n) \sum_{\substack{d | n \\ d \leq n ^ \frac{1}{k}}} h(d) \label{Theorem 1.1}
\end{align}
as $x$ tends to infinity.
\end{theorem}

This can be seen as saying that \eqref{main conj} holds ``on the average" over squarefree $n$ for a particular class of functions $h$. In view of Lemma 1 in \cite{ALLADI1989183}, it should be possible to show that the class of functions considered in Theorem \ref{thm:constant on primes case} are the sufficient case for \textit{all} choices of functions $h$ satisfying $0 \leq h(p) \leq \frac{1}{k - 1}$ for every prime $p$. We accomplish this task, but surprisingly, it was a slightly bigger undertaking in the averaged-out setting. Hence, we give it the title of theorem. 

\begin{theorem} 
\label{thm:Lemma 1 analog}
The quantity
    $$ R_{k,x}(h) = \left(\sum_{n \leq x} \mu^2(n) \sum_{\substack{d | n \\ d \leq n ^ \frac{1}{k}}} h(d) \right) \left(\sum_{n \leq x} \mu^2(n) \sum_{d|n} h(d) \right)^{-1}$$
    decreases as $h$ increases. 
\end{theorem}

This is indeed the analogue for Lemma 1 in \cite{ALLADI1989183}, and so the bigger $h$ is, the ``worse" the inequality \eqref{Theorem 1.1} is. This shows that the case considered by Theorem \ref{thm:constant on primes case} is indeed the sufficient case. 

Following the strategy of proof for \cite{ALLADI1989183}, we prove Theorem \ref{thm:Lemma 1 analog} by isolating the effect that the value of $h$ at one specific prime $p$ has on the quantity $R_{k,x}(h)$. The authors show that increasing the value of $h(p)$ decreases the value of $R_{k,x}(h)$, and so increasing the values of $h$ at each prime successively decreases the value of the $R_{k,x}(h)$. One might suspect the claim to be easier to prove in this averaged setting, but there is a factorization structure that AEV utilize which makes their proof terse. Once averaging, this structure vanishes. 

The paper is structured as follows. We first establish some preliminary propositions in Section \ref{sec:propositions}. These allow the proofs of the main theorems provided in Section \ref{sec:Proofs of Main Theorems} to be condensed. In Section \ref{sec:The Way Forward}, we discuss how one might tackle proving the entire conjecture. For smaller results unrelated to the main proofs, see the Appendix in \ref{sec:Appendix}.

\section{Acknowledgments} The author thanks Professor Krishnaswami Alladi for taking the time to look through the manuscript of someone who is no longer his student. 

\section{Propositions}
\label{sec:propositions}
 
 We will prove Theorem \ref{thm:constant on primes case} by simply evaluating the sums in the left-hand and right-hand sides of \eqref{Theorem 1.1} and comparing. For this, we require a means to evaluate several standard number theoretic sums, the first of which follows from \cite{selberg1954note}. 

\begin{proposition}
\label{prop:Selberg sums}
For $z \in \mathbb{C}$ with $\mathrm{Re}(z) > 0$, we have 
$$\sum_{n \leq x} z^{\omega(n)} \mu^2(n) = (1 + o(1)) \frac{x\log^{z - 1}{x}}{\Gamma(z)} f_0(z),$$
$$\sum_{n \leq x} z^{\omega(n)} \mu^2(n) g(n) = (1 + o(1)) \frac{x\log^{z - 1}{x}}{\Gamma(z)} f_1(z)$$
where 
$$g(m) = \prod_{p | m} \frac{p}{p + 1},$$
and
$$f_0(z) = \prod_{p} \left(1 + \frac{z}{p}\right)\left(1 - \frac{1}{p}\right)^z$$
$$f_1(z) = \prod_{p} \left(1 + \frac{z}{p + 1}\right)\left(1 - \frac{1}{p}\right)^z$$
\end{proposition} 

\begin{proof} We'll evaluate the expression involving $f_1$ as evaluating the expression involving $f_0$ can be done in an almost identical manner. Let $\beta(n) = z^{\omega(n)} \mu^2(n) g(n)$ and 
$$\mathcal{B}(s) := \sum_{n = 1}^{\infty} \frac{\beta(n)}{n^s}$$
We can see that $\beta(p^a) = z\frac{p}{p + 1}$ when $a = 1$ and $0$ if $a \geq 2$. Thus, $\mathcal{B}(s)$ admits the Euler product
\begin{align*}
  \mathcal{B}(s) &= \prod_{p} \left(1 + \frac{z \frac{p}{p + 1}}{p^s}\right) \\
\end{align*}
which converges if $\mathrm{Re}(s) > 0$ by comparison to the product considered by \cite{selberg1954note}
\begin{align*}
   \prod_{p} \left(1 + \frac{z }{p^s - 1}\right). \\
\end{align*}
We have 
\begin{align*}
  \mathcal{B}(s) &= \prod_{p} \left(1 + \frac{z \frac{p}{p + 1}}{p^s}\right) \\
                 &= \prod_{p} \left(1 - \frac{1}{p^s}\right)^{-z} \cdot \prod_{p} \left(1 - \frac{1}{p^s}\right)^z \left(1 + \frac{z \frac{p}{p + 1}}{p^s}\right)\\
                 &=\zeta(s)^z \prod_{p} \left(1 - \frac{1}{p^s}\right)^z \left(1 + \frac{z \frac{p}{p + 1}}{p^s}\right)\\
\end{align*}
Thus, the result follows from picking a suitable contour $C$ for the inverse Mellin transform 

$$\sum_{n \leq x} \beta(n) = \frac{1}{2 \pi i} \oint_{C} \mathcal{B}(s) \frac{x^s}{s} ds \\ $$

and considering the branch point and pole of $\zeta(s)^z$ at $s = 1$ in the standard fashion.

\end{proof}

We further require the following standard estimate which will allow us to handle a different yet similarly common collection of number theoretic sums. This is a standard result, the proof of which can be found many places. We are not claiming this proof as our own. 

\begin{proposition}
\label{prop:squarefree coprime sum}

For $m,x \geq 2,$ we have 
    \begin{align*}
        \sum_{\substack{n \leq x \\ \mathrm{gcd}(m,n) = 1}} \mu^2(n) = g(m)\frac{6}{\pi^2}x + O\left(\tau(m)^\frac{2}{3}\sqrt{x}\right).
    \end{align*}
\end{proposition}

\begin{proof} Applying Mobius inversion and swapping the order of summation yields
\begin{align*}
        \sum_{\substack{n \leq x \\ \mathrm{gcd}(m,n) = 1}} \mu^2(n) 
        &= \sum_{d|m} \mu(d) \sum_{\substack{k \leq \sqrt{x} \\ \mathrm{gcd}(m,k) = 1}} \mu(k) \left\lfloor \frac{x}{dk^2} \right\rfloor.
\end{align*}
 Note that we can actually restrict the summing range to $k \leq \sqrt{\frac{x}{d}}$. The main term will come from the non-fractional part of the floor function term, which looks like 

\begin{align*}
   \sum_{d|m} \mu(d) \sum_{\substack{k \leq \sqrt{\frac{x}{d}} \\ \mathrm{gcd}(m,k) = 1}} \mu(k)  \frac{x}{dk^2} 
    &= x\sum_{d|m} \frac{\mu(d)}{d} \sum_{\substack{k \leq \sqrt{\frac{x}{d}} \\ \mathrm{gcd}(m,k) = 1}} \frac{\mu(k)}{k^2} \\
    &= x\sum_{d|m} \frac{\mu(d)}{d} \sum_{\substack{k = 1 \\ \mathrm{gcd}(m,k) = 1}}^{\infty} \frac{\mu(k)}{k^2} \\
    &+O\left(x\sum_{d|m} \frac{\mu^2(d)}{d} \sum_{k = \sqrt{\frac{x}{d}} }^{\infty} \frac{1}{k^2} \right)
\end{align*} 

This main term can be factored in the standard way as 

\begin{align*}
    x\sum_{d|m} \frac{\mu(d)}{d} \sum_{\substack{k = 1 \\ \mathrm{gcd}(m,k) = 1}}^{\infty} \frac{\mu(k)}{k^2} 
    &= x\prod_{p|m} \left(1 + \frac{1}{p}\right) \prod_{p} \left(1 - \frac{1}{p^2}\right) \\
    &= \frac{6}{\pi^2}g(m)x 
\end{align*}

which gives our standard main term, and the corresponding error can be controlled by bounding with an integral over a larger region, which gives
\begin{align*}
    x\sum_{d|m} \frac{\mu^2(d)}{d} \sum_{k = \sqrt{\frac{x}{d}} }^{\infty} \frac{1}{k^2} 
    &< x\sum_{d|m} \frac{\mu^2(d)}{d} \int_{k = \frac{\sqrt{\frac{x}{d}}}{2} }^{\infty} \frac{1}{X^2} dX \\ 
    &= 2\mathcal{E}(m)\sqrt{x}
\end{align*}
where
\begin{align*}
    \mathcal{E}(m) = \sum_{d|m} \frac{\mu^2(d)}{\sqrt{d}}
\end{align*}
Indeed, the error which comes from the aforementioned fractional parts is also $\mathcal{E}(m)\sqrt{x}$:

\begin{align*}
    \sum_{d|m} \mu^2(d) \sum_{k \leq \sqrt{\frac{x}{d}}} \mu^2(k) \left\{ \frac{x}{dk^2} \right\} &\leq \sqrt{x} \sum_{d|m} \frac{\mu^2(d)}{\sqrt{d}} \\
    &=  \mathcal{E}(m) \sqrt{x}
\end{align*} 
Since $\frac{1}{\sqrt{d}}$ is decreasing, we can bound $\mathcal{E}$ by picking a $1 \leq y \leq m$ and summing $1$ indiscriminately below $y$ and summing only $\frac{1}{\sqrt{y}}$ above $y$. This gives
\begin{align*}
    \mathcal{E}(m) 
    &\leq \sum_{d \leq y} 1 + \sum_{\substack{d|m \\ d \geq y}} \frac{1}{\sqrt{y}} \\
    &\leq y + \frac{\tau(m)}{\sqrt{y}}
\end{align*}
Picking $y = \tau(m)^{\frac{2}{3}}$, we get a bound of $\mathcal{E}(m) < 2\tau(m)^{\frac{2}{3}}$. Putting this together, we get an error of the claimed size. 
\end{proof}

We next require a small calculus result which is essential to the proof of Theorem \ref{thm:Lemma 1 analog}

\begin{proposition} 
\label{prop:calculus result}

Let $f: \mathbb{R}^+ \to \mathbb{R}^+$ be an increasing, differentiable function satisfying $f(x) = x^{o(1)}$, and define $\gamma_{N}(x) = f(x)f\left(\frac{N}{x}\right)$ for some large $N \in \mathbb{N}$. Then $\gamma_N$ has a critical point at $x = \sqrt{N}$ and is decreasing on $(\sqrt{N}, N)$.    
\end{proposition}

\begin{proof} Differentiating, we obtain
\begin{align*}
    \gamma'_N(x) 
    &= f(x)f'\left(\frac{N}{x}\right)\left(-\frac{N}{x^2}\right) + f'(x)f\left(\frac{N}{x}\right) \\ 
    &= \frac{f(x)f\left(\frac{N}{x}\right)}{x} \left(\frac{xf'(x)}{f(x)} - \frac{\frac{N}{x}f'\left(\frac{N}{x}\right)}{f\left(\frac{N}{x}\right)}\right) \\
\end{align*}
    Plugging in $x = \sqrt{N}$ clearly yields $\gamma'_N = 0$. Moreover, note that $\frac{xf'(x)}{f(x)} = x \frac{d}{dx} \log{f(x)}$ which is $o(1)$ since $f(x) = x^{o(1)}$. Thus if $x \in (\sqrt{N}, N)$, $\frac{N}{x} < x$, so $\frac{xf'(x)}{f(x)} < \frac{\frac{N}{x}f'\left(\frac{N}{x}\right)}{f\left(\frac{N}{x}\right)}$, forcing $\gamma'_N$ to be negative on this interval.
\end{proof}

\section{Proofs of Main Theorems}
\label{sec:Proofs of Main Theorems}

\begin{proof}[Proof of Theorem 1.1]

For square-free $n$, we have $h(n) = c^{\omega(n)}$. Noting that all the divisors of a squarefree number must themselves be squarefree, we expand out the left hand side of \eqref{Theorem 1.1} and use Proposition \ref{prop:squarefree coprime sum} to get 
\begin{align}
    \sum_{n \leq x} \mu^2(n)\sum_{d|n} h(d) 
    &= \sum_{d \leq x} h(d) \mu^2(d) \sum_{\substack{d | n \\ n \leq x}} \mu^2(n) \nonumber \\
    &= \sum_{d \leq x} h(d) \mu^2(d) \sum_{m \leq \frac{x}{d}} \mu^2(dm) \nonumber \\
    &= \sum_{d \leq x} h(d) \mu^2(d) \sum_{\substack{m \leq \frac{x}{d} \\ \mathrm{gcd}(d,m) = 1}} \mu^2(m) \nonumber \\
    &= \frac{6}{\pi^2}x\sum_{d \leq x} \frac{\mu^2(d) h(d)g(d)}{d} + O\left(\sqrt{x}\sum_{d \leq x} \frac{\tau(d)^\frac{2}{3}\mu^2(d)h(d)}{\sqrt{d}}\right) \label{LHS of Theorem 1.1} 
\end{align}
Expanding out the right hand side of \eqref{Theorem 1.1}, we similarly have 
\begin{align}
    \sum_{n \leq x} \mu^2(n) \sum_{\substack{d | n \\ d \leq n ^ \frac{1}{k}}} h(d) \nonumber
    &= \sum_{d \leq x^{\frac{1}{k}}} h(d)\mu^2(d) \sum_{\substack{d^k \leq n \leq x \\ d|n}} \mu^2(n) \\
    &= \sum_{d \leq x^{\frac{1}{k}}} h(d)\mu^2(d) \sum_{d^k \leq md \leq x} \mu^2(md) \nonumber \\
    &= \sum_{d \leq x^{\frac{1}{k}}} h(d)\mu^2(d) \sum_{\substack{d^{k - 1} \leq m \leq \frac{x}{d} \\ \mathrm{gcd}(m,d) = 1}} \mu^2(m) \nonumber \\
    &= \frac{6}{\pi^2}x\sum_{d \leq x^{\frac{1}{k}}} \left(\frac{h(d)\mu^2(d)g(d)}{d}\right) - \frac{6}{\pi^2}\sum_{d \leq x^{\frac{1}{k}}} g(d)h(d)\mu^2(d) d^{k - 1} \label{RHS of Theorem 1.1}
\end{align}
\begin{align*}
    &+ O\left(\sum_{d \leq x^{\frac{1}{k}}} \left(\sqrt{\frac{x}{d}} - d^{\frac{k - 1}{2}}\right)h(d)\mu^2(d)\tau(d)^\frac{2}{3} \right). 
\end{align*}

To evaluate these sums, we apply the standard Abel summation formula and Proposition \ref{prop:Selberg sums} and obtain
\begin{align*}
    \sum_{d \leq x} \frac{g(d)h(d)\mu^2(d)}{d} 
    &= \left(\frac{f_1(c)}{\Gamma(c)} + o(1)\right)x\log^{c - 1}(x) \cdot \frac{1}{x} + \int_{1}^{x} \frac{\left(\frac{f_1(c)}{\Gamma(c)} + o(1)\right)X\log^{c - 1}(X)}{X^2} dX \\ 
    &=\left(\frac{f_1(c)}{\Gamma(c)} + o(1)\right)\log^{c - 1}(x) + \left(\frac{f_1(c)}{\Gamma(c)} + o(1)\right) \int_{1}^{x} \frac{\log^{c - 1}(X)}{X} dx \\ 
    &= \left(\frac{f_1(c)}{c\Gamma(c)} + o(1)\right)\log^{c}(x) \\ 
\end{align*}
and 
\begin{align*}
    \sum_{d \leq x} h(d)g(d)\mu^2(d)d^{k - 1} 
    &= x^{k - 1}\left(\frac{f_1(c)}{\Gamma(c)} + o(1)\right)x \log^{c - 1}(x) \\
    &- \int_{1}^{x} \left(\frac{f_1(c)}{\Gamma(c)} + o(1)\right)X \log^{c - 1}(X)(k - 1) X^{k - 2} dX \\ 
    &= \left(\frac{f_1(c)}{\Gamma(c)} + o(1)\right)x^{k}\log^{c - 1}{x} \\
    &- \left(\frac{(k - 1)f_1(c)}{\Gamma(c)} + o(1)\right)\int_{1}^{x} \log^{c - 1}(X) X^{k - 1} dX\\ 
    &= \left(\frac{f_1(c)}{k\Gamma(c)} + o(1)\right)x^{k}\log^{c - 1}{x} 
\end{align*}
where the last deduction follows from 
$$\int_{1}^{x} \log^{c - 1}(X) X^{k - 1} dX \sim \frac{1}{k}x^{k}\log^{c - 1}(x).$$

We now need to show that the two sums in our error terms are sufficiently smaller in magnitude. Because of the restricted range, the sum in the error term in \eqref{RHS of Theorem 1.1} is easily shown to be $O(x^\delta)$ for some $0 < \delta < 1$. More specifically, we have the following standard bounds for each of the arithmetic functions:
\begin{align}
    h(p) \leq \frac{1}{k - 1}, \tau(p) = 2, g(p) < \frac{3}{2} \label{pointwise bound for functions}
\end{align}

This and multiplicativity show that the product of the arithmetic functions in the error of \eqref{RHS of Theorem 1.1} are one-bounded for $k \geq 3$, hence

\begin{align*}
    O\left(\sum_{d \leq x^{\frac{1}{k}}} \left(\sqrt{\frac{x}{d}} - d^{\frac{k - 1}{2}}\right)h(d)\mu^2(d)\tau(d)^\frac{2}{3} \right) = O\left( x^{\frac{1}{k}} x^{\frac{1}{2}} \right)
\end{align*}
and we can take $\delta = \frac{1}{2} + \frac{1}{k}$ which is less than 1 for $k \geq 3$. With main terms of size $x\log^{O(1)}(x)$, we can handily discard this error. 

The error term in \eqref{LHS of Theorem 1.1} is being summed over a much larger range, and so it cannot be bounded as crudely. First, notice that, because our summand is only supported on squarefree numbers, we can use $\tau(d)^{\frac{2}{3}} = \left(2^\frac{2}{3}\right)^{\omega(d)}$ for squarefree $d$, making our summand

$$\frac{\left(2^\frac{2}{3}c\right)^{\omega(d)}\mu^2(d)}{\sqrt{d}}.$$
Letting $t_c = 2^\frac{2}{3}c$ and applying Proposition \ref{prop:Selberg sums} and Abel Summation, we have 

\begin{align*}
    \sum_{d \leq x} \frac{t_c^{\omega(d)}\mu^2(d)}{\sqrt{d}} 
    &=  \frac{1}{\sqrt{x}} \left(\frac{f_0(t_c)}{\Gamma(t_c)} + o(1)\right) x \log^{t_c - 1}(x) \\
    &+ \left(\frac{f_0(t_c)}{2\Gamma(t_c)} + o(1)\right)\int_{1}^{x} X\log^{t_c - 1}(X)\frac{1}{X^{\frac{3}{2}}}dX\\ 
    &= O\left(\sqrt{x}\log^{t_c - 1}(x)\right),
\end{align*}

and so the error from this sum is of size $O(x\log^{t_c - 1}{x})$. Our main terms are of order $x\log^{c}(x)$, so as long as $t_c - 1 < c$, this error is smaller than our main term. We solve algebraically to get that $c$ must satisfy $c < \frac{1}{2^{\frac{2}{3}} - 1} \approx 0.70241$. Since $c < \frac{1}{k - 1}$, as long as $k \geq 3$, our error term is small, as desired. Putting all these estimates together  yields
\begin{align*}
    x \sum_{d \leq x^{\frac{1}{k}}} \frac{h(d)g(d)\mu^2(d)}{d} - \sum_{d \leq x^{\frac{1}{k}}} h(d)g(d)\mu^2(d)d^{k - 1} 
    &= \frac{6}{\pi^2}\left(\frac{f_1(c)}{c\Gamma(c)} + o(1)\right)x\log^{c}(x^{\frac{1}{k}}) \\
    &= \frac{6}{\pi^2}\left(\frac{f_1(c)}{ck^{c}\Gamma(c)} + o(1)\right)x\log^{c}(x)
\end{align*}
for the RHS, and 
\begin{align}
    x \sum_{d \leq x} \frac{g(d)h(d)\mu^2(d)}{d} = \frac{6}{\pi^2}\left(\frac{f_1(c)}{c\Gamma(c)} + o(1)\right)x\log^{c}{x}. \label{LHS of Theorem 1.1 - evaluated}
\end{align}
for the LHS. Thus, we have 
\begin{align*}
  \sum_{n \leq x} \mu^2(d) \sum_{d|n} h(d) 
    &= (k^c + o(1))\sum_{n \leq x} \mu^2(d) \sum_{\substack{d | n \\ d \leq n ^ \frac{1}{k}}} h(d), \\
\end{align*}
and the constant 2 in Theorem \ref{thm:constant on primes case} follows from  
$$2 = \max_{\substack{k \geq 2 \\ 0 < c < \frac{1}{k - 1} }} k^c$$

\end{proof}

\begin{proof}[Proof of Theorem 1.2] For a fixed prime $p$, we have 

\begin{align*}
    \sum_{n \leq x} \mu^2(n) \sum_{d|n} h(d) 
    &= \sum_{d \leq x} \mu^2(d)h(d)\sum_{\substack{n \leq x \\ d|n}} \mu^2(n) \\
    &= \sum_{\substack{d \leq x \\ p|d}} \mu^2(d)h(d)\sum_{\substack{n \leq x \\ d|n}} \mu^2(d) + \sum_{\substack{d \leq x \\ p \nmid d}} \mu^2(d)h(d)\sum_{\substack{n \leq x \\ d|n}} \mu^2(n) \\
    &= \sum_{\substack{mp \leq x}} \mu^2(mp)h(mp)\sum_{\substack{n \leq x \\ mp|n}} \mu^2(d) + \sum_{\substack{d \leq x \\ p \nmid d}} \mu^2(d)h(d)\sum_{\substack{n \leq x \\ d|n}} \mu^2(n) \\
    &= h(p)\sum_{\substack{m \leq \frac{x}{p} \\ p\nmid m}} \mu^2(m)h(m)\sum_{\substack{n \leq x \\ mp|n}} \mu^2(n) + \sum_{\substack{d \leq x \\ p \nmid d}} \mu^2(d)h(d)\sum_{\substack{n \leq x \\ d|n}} \mu^2(n), \\
\end{align*}

and by the same argument,

\begin{align*}
    \sum_{n \leq x} \mu^2(n) \sum_{\substack{d|n \\ d \leq n^{\frac{1}{k}}}} h(d) 
    &= \sum_{d \leq x^{\frac{1}{k}}} \mu^2(d)h(d)\sum_{\substack{d^k \leq n \leq x \\ d|n}} \mu^2(n) \\
    &= h(p)\sum_{\substack{m \leq \frac{x^{\frac{1}{k}}}{p} \\ p\nmid m}} \mu^2(m)h(m)\sum_{\substack{(mp)^k \leq n \leq x \\ mp|n}} \mu^2(n) + \sum_{\substack{d \leq x^{\frac{1}{k}} \\ p \nmid d}} \mu^2(d)h(d)\sum_{\substack{d^k \leq n \leq x \\ d|n}} \mu^2(n) \\
\end{align*}

Note that neither of these expressions depend on the value of $h$ at $p$ other than instance of $h(p)$ at the front. Thus, we essentially have an expression of the form 
$$R_{k,x}(h) = \frac{Ah(p) + B}{Ch(p) + D}$$
By a simple calculus argument, this quantity will decrease as $h(p)$ increases so long as $AD - BC < 0$. We can evaluate each of these 4 sums by effectively the same means and Proposition \ref{prop:squarefree coprime sum}. For a prime $p$, let 
$$H(x,h,p) = \sum_{\substack{j \leq x \\ p \nmid j}} \frac{\mu^2(j)g(j)h(j)}{j},$$
We have 

\begin{align*}
    A: \sum_{\substack{m \leq \frac{x^{\frac{1}{k}}}{p} \\ p\nmid m}} \mu^2(m)h(m)\sum_{\substack{(mp)^k \leq n \leq x \\ mp|n}} \mu^2(n) &= \frac{6}{\pi^2}x\frac{g(p)}{p}H\left(\frac{x^{\frac{1}{k}}}{p}, h, p\right) + \text{error} \\
    B: \sum_{\substack{d \leq x^{\frac{1}{k}} \\ p \nmid d}} \mu^2(d)h(d)\sum_{\substack{d^k \leq n \leq x \\ d|n}} \mu^2(n) &= \frac{6}{\pi^2}xH(x^{\frac{1}{k}}, h, p) + \text{error} \\
    C: \sum_{\substack{m \leq \frac{x}{p} \\ p\nmid m}} \mu^2(m)h(m)\sum_{\substack{n \leq x \\ mp|n}} \mu^2(n) &= \frac{6}{\pi^2}x\frac{g(p)}{p}H\left(\frac{x}{p}, h, p\right) + \text{error}\\
    D: \sum_{\substack{d \leq x \\ p \nmid d}} \mu^2(d)h(d)\sum_{\substack{n \leq x \\ d|n}} \mu^2(n) &= \frac{6}{\pi^2}xH(x,h,p) + \text{error}\\
\end{align*}

where, for simplicity's sake for now, we have condensed into ``error" all of the terms we suspect to not make up the main terms.

The errors in $A$ and $B$ can be handled similarly to each other and to how the shorter sum was computed in the proof of Theorem \ref{thm:constant on primes case}: by applying \eqref{pointwise bound for functions} and showing that the $(mp)^k$ term (resp. $d^k$ term) doesn't contribute significantly to the main term of $A$ (resp. $B$). We'll evaluate $A$ directly, and the evaluation of $B$ transpires similarly. We have 
\begin{align*}
     A &= \frac{6}{\pi^2}x\frac{g(p)}{p}H\left(\frac{x^{\frac{1}{k}}}{p}, h, p\right) -  \frac{6}{\pi^2}p^{k-1}g(p)\sum_{\substack{m \leq \frac{x^{\frac{1}{k}}}{p} \\ p \nmid m}}\mu^2(m)h(m)g(m)m^{k-1} \\ &+ O\left(\sum_{\substack{m \leq \frac{x^{\frac{1}{k}}}{p} \\ p \nmid m}} \left(\sqrt{\frac{x}{m}} - d^\frac{k - 1}{2}\right)\frac{\mu^2(m)h(m)\tau(m)^\frac{2}{3}}{\sqrt{m}}\right) \\
\end{align*}
Using the bounds \eqref{pointwise bound for functions}, we can get a one-bound for the product of the arithmetic functions here and bound the error as in the proof of Theorem \ref{thm:constant on primes case}:
\begin{align*}
    \sqrt{x}\sum_{\substack{m \leq \frac{x^{\frac{1}{k}}}{p} \\ p \nmid m}} \frac{\mu^2(m)h(m)\tau(m)^\frac{2}{3}}{\sqrt{m}} 
    &= O(x^\frac{1}{2}x^{\frac{1}{k}}) \\
\end{align*}
which is $o(x)$ for $k \geq 3$. We will see that our main term will at worst be proportionate to $x$, making this sufficient.

For the $m^{k-1}$ sum, we want to show that it doesn't contribute significantly to the main term, or quantitatively that  
$$\sum_{\substack{m \leq \frac{x^{\frac{1}{k}}}{p} \\ p \nmid m}}\mu^2(m)h(m)g(m)m^{k-1} = o\left(xH\left(\frac{x^{\frac{1}{k}}}{p}, h, p\right)\right)$$
To do this, we split into two cases: 

\begin{enumerate}
    \item $\lim_{x \to \infty} H(x,h,p) = \infty$
    \item $\lim_{x \to \infty} H(x,h,p)$ converges to a finite value.
\end{enumerate}

Assume we're in case 1. We simply employ \eqref{pointwise bound for functions} once again. 
\begin{align*}
    \sum_{\substack{m \leq \frac{x^{\frac{1}{k}}}{p} \\ p \nmid m}}\mu^2(m)h(m)g(m)m^{k-1} 
    &\leq \sum_{\substack{m \leq \frac{x^{\frac{1}{k}}}{p}}}m^{k-1} \\
    &= \left(\frac{1}{kp^k} + o(1)\right)x
\end{align*}
If $H$ tends to infinity with $x$, then the factor of $H$ in our main term makes it slightly larger than linear. Hence, under this condition, the $m^{k-1}$ sum doesn't contribute significantly to the main term.

Now, assume we're in case 2. Our main term is now directly proportionate to $x$, requiring our error to be $o(x)$. It's a result of Halasz \cite{halasz1968mittelwerte} that if $f$ is a multiplicative function supported on squarefree numbers and $\sum_{p} \frac{f(p)}{p}$ converges, then $\sum_{n \leq x} f(n) = o(x)$. Letting $f(n) = \mathbf{1}_{p \nmid n} \mu^2(n)g(n)h(n)$, the Case 2 hypothesis shows that $\sum_{n} \frac{f(n)}{n}$ converges, and therefore so must the sum over primes. This together with Abel summation handily shows that, under case 2, 
$$\sum_{\substack{m \leq \frac{x^{\frac{1}{k}}}{p} \\ p \nmid m}}\mu^2(m)h(m)g(m)m^{k-1} = o(x)$$

To rigorously evaluate $C$ and $D$, we'll evaluate $D$, and $C$ will transpire similarly. Following in the footsteps of the proof of Theorem 1.1, we have 
\begin{align*}
    \sum_{\substack{d \leq x \\ p \nmid d}} \mu^2(d)h(d)\sum_{\substack{n \leq x \\ d|n}} \mu^2(n) &= \sum_{\substack{d \leq x \\ p \nmid d}} \mu^2(d)h(d)\sum_{\substack{m \leq \frac{x}{d} \\ \gcd(d,m) = 1}} \mu^2(m) \\
    &= \frac{6}{\pi^2}x\sum_{\substack{d \leq x \\ p \nmid d}} \frac{\mu^2(d)h(d)g(d)}{d} + O\left(\sqrt{x} \sum_{\substack{d \leq x \\ p \nmid d}} \frac{\tau(d)^{\frac{2}{3}}\mu^2(d)h(d)}{\sqrt{d}}\right) \\
    &= \frac{6}{\pi^2}xH(x,h,p) + O\left(\sqrt{x} \sum_{\substack{d \leq x \\ p \nmid d}} \frac{\tau(d)^{\frac{2}{3}}\mu^2(d)h(d)}{\sqrt{d}}\right)
\end{align*}
Similarly to the evaluation of $A$, we need to show that 
$$  \sum_{\substack{d \leq x \\ p \nmid d}} \frac{\tau(d)^{\frac{2}{3}}\mu^2(d)h(d)}{\sqrt{d}} = o\left(\sqrt{x}H(x,h,p)\right)$$

We split into the same two cases. For case 1, we proceed by invoking \eqref{pointwise bound for functions} once more to obtain
\begin{align*}
    \sum_{\substack{d \leq x \\ p \nmid d}} \frac{\tau(d)^{\frac{2}{3}}\mu^2(d)h(d)}{\sqrt{d}}
    &\leq \sum_{\substack{d \leq x}} \frac{1}{\sqrt{d}} \\
    &=(2 + o(1)) \sqrt{x},
\end{align*}
so again, we are saving over the main term by a factor of $H(x,h,p)$, which in case 1 is sufficient.

Assuming we're in case 2, we proceed almost identically to before except now we have to deal with the $\tau^{2/3}(d)$ term. We know that $\sum_{p} \frac{\mu^2(p)g(p)h(p)}{p}$ converges, and since $\frac{\tau^{2/3}(p)}{g(p)} \leq \frac{3 \cdot 2^{2/3}}{2}$, the sum we must consider, namely $\sum_{p} \frac{\mu^2(p)\tau^{2/3}(p)h(p)}{p}$, is at most a constant multiple larger, and hence it is also convergent. Thus, under case 2, Halasz' result applies. It, together with Abel summation, yields
$$\sum_{\substack{d \leq x \\ p \nmid d}} \frac{\tau(d)^{\frac{2}{3}}\mu^2(d)h(d)}{\sqrt{d}} = o\left(\sqrt{x}\right)$$

Putting this altogether, we finally have 
$$AD - BC = \left(\frac{6}{\pi^2} + o(1)\right)\frac{g(p)}{p}x \left(H\left(\frac{x^{\frac{1}{k}}}{p},h,p\right) H(x,h,p) - H\left(x^{\frac{1}{k}},h,p\right) H\left(\frac{x}{p}\right)\right)$$
In view of \eqref{pointwise bound for functions} and in comparison to the $x$-th harmonic sum, we can clearly see that $H(x,h,p) = O(\log(x)) = x^{o(1)}$. It's certainly increasing, but in order to apply Proposition \ref{prop:calculus result}, we need differentiability, and $H$ is defined in a discrete manner which makes it discontinuous. We can define $\mathcal{H}(x,h,p)$ to be the function which is extended to all positive real numbers, is increasing, continuous and differentiable, and agrees with $H(x,h,p)$ at positive integer values of $x$. This is possible since we interpolate only countably many values. Applying Proposition \ref{prop:calculus result} with $N = \frac{x^{1 + \frac{1}{k}}}{p}$ and $f(x) = \mathcal{H}(x,h,p)$, we have 
$$AD - BC = \left(\frac{6}{\pi^2} + o(1)\right)x\frac{g(p)}{p} \left(\gamma_N(x) - \gamma_N\left(\frac{x}{p}\right)\right)$$
Both $x$ and $\frac{x}{p}$ are greater than $\sqrt{N}$, and so $\gamma_N$ decreases from  $\frac{x}{p}$ to $x$. Thus, 
$\gamma_N(x) - \gamma_N\left(\frac{x}{p}\right)$ is negative, and so too is $AD - BC$ for sufficiently large $x$.

\end{proof}

\section{The Way Forward} 
\label{sec:The Way Forward}

It still remains to show the original conjecture of AEV, namely \eqref{main conj}. Though Soundararajan's result is stronger, AEV's initial analysis seems to provide a more transparent means for future improvement, as evidenced by the many analogies between their paper and this one. To see how to improve, it is important to note how they arrived at their conjecture. 

Loosely speaking, at one point in their proof, AEV write a natural number $n$ as the product of $k$ natural numbers $d_1,...d_k$, i.e. $n = d_1 \cdot ... \cdot d_k$. They note that by a simple Pigeonhole Principle argument, at least one of the $d_i$ must satisfy $d_i \leq n^{1/k}$. However, on average, one should expect roughly half, or $k/2$, of the $d_i$ to satisfy $d_i \leq n^{1/k}$ (this is a sort of ``symmetry point" for such a product). This change in a factor of $k/2$ is where they anticipate most of the loss in their method comes from (indeed: $\frac{2k}{k/2} = 4$). Thus, if one could somehow quantify this heuristic in a meaningful way, it could be used to resolve the conjecture fully. 

Define

\[g_k(n) := \sum_{d_1 \cdot ... \cdot d_k = n} |\{ i \text{ }|\text{ }d_i \leq n^{1/k} \}| \]

This quantity can be seen as the mean number of ``small" divisors present in a random $k$-fold product representation of $n$. The trivial bound based on these sort of Pigeonhole Principle arguments is 

\[\tau_k(n) \leq g_k(n) \leq (k - 1)\tau_k(n)\]
where, as usual,
\[\tau_k(n) := \sum_{d_1 \cdot ... \cdot d_k = n} 1\]
is the $k$-fold divisor function. Given the previously stated heuristic, one might expect 

\begin{align}
    g_k(n) = \left(\frac{k}{2} + o(1)\right) \tau_k(n) \label{conj for $g_k$}
\end{align}
where, say, $o(1) \to 0$ as $\omega(n) \to \infty$. Though there isn't an argument presently for whether \eqref{conj for $g_k$} directly implies \eqref{main conj}, it certainly seems like a useful object of study, if only morally speaking. Improving on the trivial bound given current techniques (and those explored in \cite{ALLADI1989183}) seems challenging. A new idea is likely needed.

\section{Appendix}
\label{sec:Appendix}

It requires a bit of extra work to show that \eqref{main conj} actually implies \eqref{Theorem 1.1}. Suppose \eqref{main conj} held. Let $f_k$ be a function tending to 0 as its input tends to infinity so that we may rewrite \eqref{main conj} as 
\begin{align*}
\sum_{d|n} h(d) \leq (4 + f_k(\omega(n))) \sum_{\substack{d|n \\ d \leq n^{\frac{1}{k}}}} h(d)   
\end{align*}
The constants in the main terms will be different, but summing over squarefree $n \leq x$, it essentially suffices to show 
\begin{align*}
     \sum_{n \leq x} \mu^2(d)f_k(\omega(n)) \sum_{\substack{d|n \\ d \leq n^{\frac{1}{k}}}} h(d)
     &= o\left(\sum_{n \leq x} \mu^2(d) \sum_{\substack{d|n \\ d \leq n^{\frac{1}{k}}}} h(d) \right).
\end{align*}
Expanding gives
\begin{align*}
    \sum_{n \leq x} \mu^2(d)f_k(\omega(n)) \sum_{\substack{d|n \\ d \leq n^{\frac{1}{k}}}} h(d) &= \sum_{d \leq x^{\frac{1}{k}}} h(d) \mu^2(d) \sum_{\substack{d^k \leq n \leq x \\ d|n}} \mu^2(n) f_k(\omega(n)) \\
    &= \sum_{d \leq x^{\frac{1}{k}}} h(d) \mu^2(d) \sum_{d^{k - 1} \leq m \leq \frac{x}{d}} \mu^2(md) f_k(\omega(md)) \\
    &\leq \sum_{d \leq x^{\frac{1}{k}}} h(d) \mu^2(d) \sum_{m \leq \frac{x}{d}} f_k(\omega(m)) \\ 
\end{align*}
(Note: $f_k(\omega(md) \leq f_k(\omega(m) + \omega(d)) \leq f_k(\omega(m)$). If we had $\sum_{m \leq y} f_k(\omega(m)) = o(y)$, this last sum would then become 
\begin{align*}
    & o\left(x \sum_{d \leq x^{\frac{1}{k}}} \frac{h(d)\mu^2(d)}{d}\right) \\
    &=o\left( x \log^{c}(x)\right)
\end{align*}
by Proposition \ref{prop:Selberg sums} and partial summation. From the analysis in Section \ref{sec:Proofs of Main Theorems}, we know that both the left hand and right hand sums in \eqref{Theorem 1.1} have a main term of size $x \log^{c}(x)$, and so it remains to show $\sum_{m \leq y} f_k(\omega(m)) = o(y)$.

We implement the Erdos-Kac theorem, found in \cite{erdos1940gaussian}:

$$\lim_{x \to \infty} \frac{1}{x} |\{n \leq x : a \leq \frac{\omega(n) - \log\log(n)}{\sqrt{\log\log(n)}} \leq b\}| = \Phi(a,b)$$
where $\Phi(a,b)$ is the Gaussian distribution function. Pick a slow-growing function $a(x)$ (for our purposes, any choice with $a(x) = o(\sqrt{\log\log(x)})$ will do), and let 
$$A(x) = \{n \leq x : -a(x) \leq \frac{\omega(n) - \log\log(n)}{\sqrt{\log\log(n)}} \leq a(x)\}$$
We have 
\begin{align*}
    \sum_{m \leq x} f_k(\omega(m)) 
    &= \sum_{\substack{m \leq x \\ m \notin A(x)}} f_k(\omega(m)) + \sum_{\substack{m \leq x \\ m \in A(x)}} f_k(\omega(m)) \\
    &\leq \sum_{m \leq x} f_k\left(\log\log(n) - a(x)\sqrt{\log\log{n}}\right) + O_k\left(x(1 - \Phi(-a(x), a(x))\right) \\
    &= o(x).
\end{align*}

We also fall into a neat proof of a relation between $f_0$ and $f_1$, the functions from Proposition \ref{prop:Selberg sums}:

\begin{proposition} For $c > 0$, we have 
    $$f_1(c) = \frac{\pi^2}{6}f_0(1+c)$$
\end{proposition}
There is a way of deriving this from just the product expressions for $f_0$ and $f_1$ themselves, but the following argument still seemed interesting. 
\begin{proof} Since the following divisor sum is complete in $d$ (and $n$ is squarefree), we can write 
\begin{align*}
    \sum_{d|n} h(d) 
    &= \sum_{j = 1}^{\omega(n)} {\omega(n) \choose j} c^{j} \\
    &= (1 + c)^{\omega(n)} \\
\end{align*}
and then applying Proposition \ref{prop:Selberg sums} becomes a lot more straightforward:
\begin{align*}
    \sum_{n \leq x} \mu^2(n) \sum_{d|n} h(d) 
    &= \sum_{n \leq x} \mu^2(n) (1 + c)^{\omega(n)} \\
    &= \left(\frac{f_0(1 + c)}{\Gamma(1 + c)} + o(1)\right)x\log^{c}(x) \\
    &= \left(\frac{f_0(1 + c)}{c\Gamma(c)} + o(1)\right)x\log^{c}(x) \\
\end{align*}
The claim follows from comparing this to \eqref{LHS of Theorem 1.1 - evaluated}.
\end{proof}

\bibliographystyle{alpha}
\bibliography{mylib}

\end{document}